\documentclass[10pt]{article}
\usepackage[all]{xy}
\usepackage{amsfonts,amsmath,oldgerm,amssymb,amscd}
\newcommand{\ra}{\rightarrow}   

\newcommand{\by}[1]{\stackrel{#1}{\ra}}

\newcommand{\surj}{\ra\!\!\!\ra}	

\newcommand{\ol}{\overline}		\newcommand{\wt}{\widetilde}
\newcommand{\iso}{\by \sim}

\newtheorem{theorem}{Theorem}[section]
\newtheorem{proposition}[theorem]{Proposition}
\newtheorem{lemma}[theorem]{Lemma}
\newtheorem{definition}[theorem]{Definition}
\newtheorem{corollary}[theorem]{Corollary}

\newcommand{\ga}{\alpha}		\newcommand{\gb}{\beta}

\newcommand{\gj}{\blacksquare}

	\newcommand{\BN}{\mbox{$\mathbb N$}}

\newcommand{\CM}{\mbox{$\mathcal M$}}

\newcommand{\op}{\mbox{$\oplus$}}

\newcommand{\Um}{\mbox{\rm Um}}		\newcommand{\SL}{\mbox{\rm SL}}
\newcommand{\GL}{\mbox{\rm GL}}		\newcommand{\ot}{\mbox{$\otimes$}}
\newcommand{\Aut}{\mbox{\rm Aut}}

\begin{document} 
\begin{center}
{\Large \bf Cancellation of projective modules over non-Noetherian rings
        }\\\vspace{.2in} {\large 
         Manoj K. Keshari}
\footnote{Department of Mathematics, IIT Bombay, Mumbai 400076, India;\;
        keshari@math.iitb.ac.in 
}\end{center}


\begin{abstract}
$(i)$ Let $R$ be a ring of dimension $0$ and
  $A=R[Y_1,\ldots,Y_n,(f_1\ldots f_m)^{-1}]$, where $m\leq n$,
  $Y_1,\ldots,Y_n$ are variables over $R$ and $f_i\in R[Y_i]$. Then
  all projective $A$-modules are free and $E_r(A)$ acts transitively
  on $\Um_r(A)$ for $r\geq 3$.

$(ii)$ Let $R$ be a ring of dimension $d$ and $A$ be one of $R[Y]$ of
  $R[Y,Y^{-1}]$, where $Y$ is a variable over $R$. Let $P$ be a
  projective $A$-module of rank $\geq d+1$ satisfying property
  $\Omega(R)$ (see \ref{dd} for definition of property
  $\Omega(R)$). Then $E(A\op P)$ acts transitively on $\Um(A\op
  P)$. When $P$ is free, this result is due to Yengui: $A=R[Y]$ and
  Abedelfatah: $A=R[Y,Y^{-1}]$.
\end{abstract}

\section{Introduction}

{\it Rings are assumed to be commutative with unity and 
  modules are finitely generated. The dimension of a ring means
  its Krull dimension and projective modules are of constant
  rank.}

Let $R$ be a {\it Noetherian} ring of dimension $d$ and
$A=R[Y_1,\ldots,Y_n,(f_1\ldots f_m)^{-1}]$, where $m\leq n$,
$Y_1,\ldots,Y_n$ are variables over $R$ and $f_i\in R[Y_i]$. If $P$ is
a projective $A$-module of rank $\geq max\{2,d+1\}$, then
author-Dhorajia (\cite{AK}, Theorem 3.12) proved that $E(A\op P)$ acts
transitively on $\Um(A\op P)$. In particular $P$ is cancellative,
i.e. $P\op A^t \iso Q\op A^t$ for some projective $A$-module $Q$
$\implies P\iso Q$. The case $n=m=0$ of this result is due to Bass
\cite{Bass}, $n=1,m=0$ is due to Plumstead \cite{P},
$n=m=1$ and $f_1=Y_1$ is due to Mandal \cite{M1} (he proved that $P$
is cancellative), $m=0$ is due to Rao \cite{Ravi} (he proved
that $P$ is cancellative) and Laurent polynomial case $f_i=Y_i$ is
due to Lindel \cite{L95}.

Heitmann (\cite{H84}, Corollary 2.7) generalized Bass' result to all
commutative non-Noetherian rings.  It is natural to ask if analog of
above results hold for non-Noetherian rings.

Let $R$ be a ring of dimension $0$ and $A=R[Y_1,\ldots,Y_n]$ be a
polynomial ring in $n$ variables $Y_1,\ldots,Y_n$ over $R$. Then
Brewer-Costa \cite{BC} proved that all projective $A$-modules are
free, generalizing the well known Quillen-Suslin theorem \cite{Q,Su1}
(see Ellouz-Lombardi-Yengui \cite{ELY} for a constructive proof).
Abedelfatah \cite{AA1} generalized Brewer-Costa's result by proving
that $E_r(A)$ acts transitively on $\Um_r(A)$ for $r\geq 3$. We
generalize these results as follows (see \ref{zero},
\ref{elem}). This is non-Noetherian analog of author-Dhorajia's result
in case $d=0$.

\begin{theorem}
Let $R$ be a ring of dimension $0$ and $A=R[Y_1,\ldots,Y_n,(f_1\ldots
  f_m)^{-1}]$, where $m\leq n$, $Y_1,\ldots,Y_n$ are variables over
$R$ and $f_i\in R[Y_i]$. Then all projective $A$-modules are free and
$E_r(A)$ acts transitively on $\Um_r(A)$ for $r\geq 3$.
\end{theorem}

Let $R$ be a ring of dimension $d$ and $n\geq d+2$. Then Yengui
\cite{Y11} proved that $E_n(R[Y])$ acts transitively on $\Um_n(R[Y])$
which is non-Noetherian analog of Plumstead's result in free
case. Abedelfatah \cite{AA} proved that $E_n(R[Y,Y^{-1}])$ acts
transitively on $\Um_n(R[Y,Y^{-1}])$ which is non-Noetherian analog of
Mandal's result in free case.  We generalize both results as
follows (\ref{K3}). See (\ref{dd}) for definition of property
$\Omega(R)$.

\begin{theorem}\label{ddd}
Let $R$ be a ring of dimension $d$ and $A$ be one of $R[Y]$ or
$R[Y,Y^{-1}]$, where $Y$ is a variable over $R$. If $P$ is a
projective $A$-module of rank $\geq d+1$ satisfying property
$\Omega(R)$, then $E(A\op P)$ acts transitively on $\Um(A\op P)$. In
particular $P$ is cancellative.
\end{theorem}

We generalize (\ref{ddd}) for Pr\"ufer domain as follows (see
\ref{K4}): {\it Let $R$ be a Pr\"ufer domain of dimension $d$ and
  $A=R[Y,f^{-1}]$, where $Y$ is a variable over $R$ and $f\in
  R[Y]$. If $P$ is a projective $A$-module of rank $\geq d+1$, then
  $E(A\op P)$ acts transitively on $\Um(A\op P)$.}


\section{Preliminaries}

Let $A$ be a ring, $J$ an ideal of $A$ and $M$ an $A$-module. We say
that $m\in M$ is {\it unimodular} if there exist $\phi\in
M^*=Hom_A(M,A)$ such that $\phi(m)=1$. The set of unimodular
elements of $M$ is denoted by $\Um(M)$. We write
$\Um^1(A\op M,J)$ for the set of $(a,m)\in \Um(A\op M)$
such that $a\in 1+J$. We write $\Um(A\op M,J)$ for the set of $(a,m)\in
\Um^1(A\op M,J)$ such that $m \in JM$. We write $\Um_r(A,J)$ for
$\Um(A\op A^{r-1},J)$.

The group of $A$-automorphism of $M$ is denoted by $\Aut_A(M)$. 
We write $E^1(A\oplus M,J)$ for the subgroup of
$\Aut_A(A\oplus M)$ generated by automorphisms
$\Delta_{a\varphi}$ and $\Gamma_{m}$, where
$$\Delta_{a\varphi}=\left(
\begin{matrix}
 1 & a\varphi\\
0 & id_M
\end{matrix}  \right)
~~\mbox{and}~~ \Gamma_{m}=\left(\begin{matrix}
1&0\\
m&id_M 
\end{matrix}\right)
~~\mbox{with}~~ a\in J,\;\varphi \in M^*,\; m \in M.$$ 

We write
$E^1(A\op M)$ for $E^1(A\op M,A)$.  Let $E_{r+1}(A)$ denote the
subgroup of $\SL_{r+1}(A)$ generated by elementary matrices
$I+ae_{ij}$, where $a\in A$, $i\not= j$ and $e_{ij}$ is the matrix
with only non-zero entry $1$ at $(i,j)$-th place. We write
$E^1_{r+1}(A,J)$ for the subgroup of $E_{r+1}(A)$ generated by
$\Delta_{{\bf a}}$ and $\Gamma_{{\bf b}}$, where
$$\Delta_{{\bf a}}=\left(
\begin{matrix}
 1 & {\bf a}\\
0 & id_{F}
\end{matrix}  \right)
~~\mbox{and}~~ \Gamma_{{\bf b}}=\left(\begin{matrix}
1&0\\
{\bf b}^t & id_{F} 
\end{matrix}\right), 
~~\mbox{where}~~
F=A^r,\; {\bf a}\in JF,\; {\bf b} \in F.$$ 

Let $p\in M$ and $\varphi \in M^*$ be such
that $\varphi(m)=0$. Let $\varphi_p \in End(M)$ be defined as
$\varphi_p(q)=\varphi(q)p$. Then $1+\varphi_p$ is an
automorphism of $M$. The automorphism $1+\varphi_p$ of $M$
is called a {\it transvection} of $M$ if either $p\in \Um(M)$ or
$\varphi \in \Um(M^*)$. We write $E(M)$ for the subgroup of $\Aut(M)$
generated by transvections of $M$.

Due to following result of Bak-Basu-Rao (\cite{BBR}, theorem 3.10), we
can interchange $E(A\op P)$ and $E^1(A\op P)$.

\begin{theorem}
Let $A$ be a ring and $P$ a projective $A$-module of rank $\geq
2$. Then $E^1(A\op P)=E(A\op P)$.
\end{theorem}

The following result of Heitmann (\cite{H84}, Corollary 2.7)
generalizes Bass's cancellation \cite{Bass} to non-Noetherian rings.

\begin{theorem}\label{H}
Let $A$ be a ring of dimension $d$ and $P$ a projective $A$-module of
rank $\geq d+1$. Then $E(A\op P)$ acts transitively on $\Um(A\op
P)$. In particular $P$ is cancellative.
\end{theorem}

The following result of Brewer-Costa \cite{BC} generalizes
Quillen-Suslin theorem \cite{Q,Su1} to all zero-dimensional rings.

\begin{theorem}\label{zero1}
Let $R$ be a ring of dimension $0$ and $A=R[Y_1,\ldots,Y_n]$ a
polynomial ring in $n$ variables $Y_1,\ldots,Y_n$ over $R$.  Then all
projective $A$-modules are free.
\end{theorem}

We will state five results which are proved with
assumption that rings are Noetherian. But the same proof works for
non-Noetherian rings.

\begin{lemma}\label{lift}
(\cite{AK1}, Remark 2.2)
Let $A$ be a ring, $I$ an ideal of $A$ and $P$ a projective
$A$-module. Then the natural map $E(A\op P) \ra E(\frac{A\op P}{I(A\op P)})$
is surjective.
\end{lemma}

\begin{lemma}\label{ak3.1}
(\cite{AK}, Lemma 3.1) Let $A$ be a ring, $J$ an ideal of $A$ and $P$
  a projective $A$-module. Let ``bar'' denote reduction modulo the
  nil-radical of $A$. Assume $E^1(\ol A\op \ol
  P,\ol J)$ acts transitively on $\Um^1(\ol A \op \ol P,\ol J)$. Then
  $E^1(A\op P,J)$ acts transitively on $\Um^1(A\op P,J)$.
\end{lemma}

\begin{lemma}\label{l95}
(\cite{L95}, Lemma 1.1)
Let $A$ be a reduced ring and $P$ an $A$-module. Assume $s\in A$ is a
non-zerodivisor such that $P_s$ is free of rank $r\geq 1$. Then there
exist $p_1,\ldots,p_r \in P$, $\phi_1,\ldots,\phi_r \in P^*$ and $t\in
\BN$ such that

(i) $s^tP \subset F$ and $s^tP^* \subset G$ with $F=\sum_1^r Ap_i$ and
$G=\sum_1^r A\phi_i$.

(ii) $(\phi_i(p_j))_{1\leq i,j\leq r} =$ diagonal $(s^t,\ldots,s^t)$.
\end{lemma}

\begin{lemma}\label{note}
(\cite{AK} Lemma 3.10)
Let $A$ be a reduced ring and $P$ a projective $A$-module of rank
$r$. Assume there exist a non-zerodivisor $s\in A$ such that $P_s$ is
free.  Choose $p_1,\ldots,p_r\in P$, $\varphi_1,\ldots,\varphi_r\in
P^*$ satisfying (\ref{l95}). Let $(a,p)\in \Um(A\op P,sA)$ with
$p=c_1p_1+\ldots +c_rp_r$, where $c_i\in sA$ for
all $i$. Assume there exist $\phi\in E^1_{r+1}(A,sA)$ such that
$\phi (a,c_1,\ldots,c_r)=(1,0,\ldots,0)$. Then there exist $\Phi\in
E(A\op P)$ such that $\Phi(a,p)=(1,0)$.
\end{lemma}

\begin{lemma}\label{w}
(\cite{W}, Lemma 4.2) Let $A$ be a reduced ring and $P$ an
  $A$-module. Assume there exist non-zerodivisors $s_1,\ldots,s_r \in
  A$, $p_1,\ldots,p_r \in P$ and $\phi_1,\ldots,\phi_r\in P^*$ such
  that $(\phi_i(p_j))_{r \times r}=$ diagonal $(s_1,\ldots,s_r):=N$.  Let
  $\CM$ be the subgroup of $\GL_r(A)$ consisting of all matrices of the form
  $I+TN^2$ for $T\in M_r(A)$. Then the map $$\Phi : \CM\ra \Aut_A(P);
  ~~~ \Phi(I+TN^2)=
  id_P+(p_1,\ldots,p_r)\,T\,N\,(\phi_1,\ldots,\phi_r)^t$$ is a group
  homomorphism.
\end{lemma}

The following result is from Lam's book (\cite{Lam}, Proposition
VI.1.14).

\begin{proposition}\label{split}
Let $B$ be a ring and $a,b \in B$ two comaximal elements. Then for any
$\sigma \in E_n(B_{ab})$ with $n\geq 3$, there exist $\ga \in
E_n(B_b)$ and $\gb\in E_n(B_a)$ such that $\sigma= (\ga)_a(\gb)_b$.
\end{proposition}

We state Quillen-Suslin theorem \cite{Q,Su1}. Note that any
commutative ring is a filtered union of Noetherian commutative
rings. Hence following result will follow from Noetherian case.

\begin{theorem}\label{QS1}
Let $R$ be a ring and $P$ a projective $R[Y]$-module. Let $f\in R[Y]$
be a monic polynomial such that $P_f$ is free. Then $P$ is free.
\end{theorem}

We state a result of Yengui \cite{Y11} and Abedelfatah
\cite{AA} respectively.

\begin{theorem}\label{YA}
Let $A$ be a ring of dimension $d$, $Y$ a variable over $A$ and $n\geq
d+2$. Then

$(i)$ $E_n(A[Y])$ acts transitively on $\Um_n(A[Y])$.

$(ii)$ $E_n(A[Y,Y^{-1}])$ acts transitively on $\Um_n(A[Y,Y^{-1}])$.
\end{theorem}

\section{Zero dimension case}

In this section we prove our first result.

\begin{proposition}\label{zero2}
Let $\Sigma(n)$ be set of rings which is closed w.r.t. following
properties: 

$(i)$ If $R\in \Sigma(n)$ and $0\not= f\in R[Y]$ is non-unit, then
$R[Y]_{f(1+fR[Y])}\in \Sigma(n)$.

$(ii)$ If $R\in \Sigma(n)$, then all 
projective modules over $R[Y_1,\ldots,Y_n]$ are
free, where $Y_1,\ldots,Y_n$ are variables over $R$. 

Then, for $R\in \Sigma(n)$, all
 projective modules over $R[Y_1,\ldots,Y_n,(f_1\ldots f_m)^{-1}]$
are free, where $m\leq n$ and $f_i\in R[Y_i]$.
\end{proposition}

\begin{proof}
Let $P$ be a projective $A=R[Y_1,\ldots,Y_n,(f_1\ldots
  f_m)^{-1}]$-module of rank $r$.  If $m=0$, then $P$ is free by
assumption $(ii)$.  Assume $m>0$ and use induction on $m$. Write $C=
R[Y_1,\ldots,Y_n,(f_1\ldots f_{m-1})^{-1}]$,
$S=1+f_mR[Y_m]$ and $B=R[Y_m]_{f_mS}$.  Then $A=C_{f_m}$, $B\in
\Sigma(n)$ by assumption $(i)$ and
$S^{-1}A=B[Y_1,\ldots,Y_{m-1},Y_{m+1},\ldots,Y_n,(f_1\ldots
  f_{m-1})^{-1}]$. By induction on $m$, $S^{-1}P$ is
free. Since $P$ is finitely generated, we can find $g\in S$ such that
$P_g$ is free. Note that $f_m$ and $g$ are comaximal elements of
$R[Y_m]$. Consider the fiber product diagram
$$\xymatrix{ C\ar [r] \ar [d] & C_{f_m}=A \ar [d] \\ C_g \ar [r] &
  C_{f_mg}=A_g }.$$ Patching projective modules $P$ over $C_{f_m}$ and
$(C_g)^r$ over $C_g$, we get  $P \iso Q_{f_m}$, where $Q$ is a projective
$C$-module of rank $r$. By induction on $m$, projective modules over $C$
are free. Hence $Q$ is free and therefore $P$ is free.  $\hfill \gj$
\end{proof}
\medskip

Let $\Sigma (n)$ be the set of rings of dimension $0$. If $R\in
\Sigma(n)$ and $0\not =f\in R[Y]$ is a non-unit, then $\dim R[Y]=1$
and $\dim R[Y]_{f(1+fR[Y])}=0$. Hence $R[Y]_{f(1+fR[Y])} \in
\Sigma(n)$. Using (\ref{zero1}), projective modules over polynomial ring
$R[Y_1,\ldots,Y_n]$ are free. Hence $\Sigma(n)$ satisfies 
hypothesis $(i,ii)$ of (\ref{zero2}). Therefore, we get the following
generalization of (\ref{zero1}).

\begin{proposition}\label{zero}
Let $R$ be a ring of dimension $0$ and $A=R[Y_1,\ldots,Y_n,(f_1\ldots
  f_m)^{-1}]$, where $m\leq n$, $Y_1,\ldots,Y_n$ are variables over
$R$ and $f_i\in R[Y_i]$. Then all projective $A$-modules are free.
\end{proposition}

\begin{theorem}\label{elem}
Let $R$ be a ring of dimension $0$ and $A=R[Y_1, \ldots,Y_n,(f_1\ldots
  f_m)^{-1}]$, where $m\leq n$, $Y_1,\ldots,Y_n$ are variables over
$R$ and $f_i\in R[Y_i]$. Then $E_r(A)$ acts transitively on $\Um_r(A)$
for $r\geq 3$.
\end{theorem}

\begin{proof}
The case $m=0$ is due to Abedelfatah \cite{AA1}. Assume $m>0$ and use
induction on $m$. Let $v\in E_r(A)$. Write $C=
R[Y_1,\ldots,Y_n,(f_1\ldots f_{m-1})^{-1}]$,
$S=1+f_mR[Y_m]$ and $B=R[Y_m]_{f_mS}$. Then $B$ is $0$ dimensional,
$A=C_{f_m}$ and
$S^{-1}A=B[Y_1,\ldots,Y_{m-1},Y_{m+1},\ldots,Y_n,(f_1\ldots
  f_{m-1})^{-1}]$. By induction on $m$, $E_r(S^{-1}A)$ acts
transitively on $\Um_r(S^{-1}A)$.  Hence there exist $\sigma \in
E_r(S^{-1}A)$ such that $\sigma(v)=e_1=(1,0,\ldots,0)$.  We can find $g\in S$ and
$\wt \sigma \in E_r(C_{f_mg})$ such that $\wt \sigma(v)=e_1 $. Note
that $f_m$ and $g$ are comaximal elements of $R[Y_m]$. Consider the
fiber product diagram
$$\xymatrix{ C\ar [r] \ar [d] & C_{f_m}=A \ar [d] \\ C_g \ar [r] &
  C_{f_mg}=A_g }$$ By (\ref{split}), $\wt \sigma$ has a splitting $\wt
\sigma=(\ga)_{f_m}(\gb)_g$, where $\ga\in E_r(C_g)$ and $\gb\in
E_r(C_{f_m})$. We have unimodular elements $\gb(v) \in \Um_r(C_{f_m})$
and $\ga^{-1}(e_1)\in \Um_r(C_g)$ whose images in $C_{f_mg}$ are
same. Hence patching $\gb(v)$ and $\ga^{-1}(e_1)$, we get $w\in
\Um_r(C)$ such that its image in $C_{f_m}$ is $\gb(v)$. By induction
on $m$, $E_r(C)$ acts transitively on $\Um_r(C)$. Hence there exist
$\phi\in E_r(C)$ such that $\phi(w)=e_1$. If $\Phi_1\in E_r(C_{f_m})$
is the image of $\phi$, then $\Phi_1(\ga(v))=e_1$. Write
$\Phi=\Phi_1\ga \in E_r(A)$, we are done.  $\hfill \gj$
\end{proof}

\section{Main Theorem}

The following result is proved in (\cite{Kes}, Lemma 3.3) with the
assumption that ring is Noetherian. Using (\ref{w}), same proof works
for non-Noetherian ring. Hence we omit the proof.

\begin{lemma}\label{k1}
Let $A$ be a reduced ring and $P$ a projective $A$-module of rank
$r$. Assume there exist a non-zerodivisor $s\in A$ such that
(\ref{l95}) holds. Assume $R^r$ is cancellative, where
$R=A[X]/(X^2-s^2X)$. Then any element of $\Um^1(A\op P,s^2A)$ can be
taken to $(1,0)$ by some element of $\Aut(A\op P,sA)$.
\end{lemma}

An immediate consequence of (\ref{k1}) is the following result. Its
proof is same as of (\cite{Kes}, Corollary 3.5) using (\ref{H}).  

\begin{corollary}\label{k11}
Let $A$ be a reduced ring of dimension $d$ and $P$ a projective
$A$-module of rank $d$. Assume there exist a non-zerodivisor
$s\in A$ such that (\ref{l95}) holds. Assume $R^d$ is cancellative,
where $R=A[X]/(X^2-s^2X)$. Then $P$ is cancellative.
\end{corollary}

Let $R$ be a ring and $I$ an ideal of $R$. For $n\geq 3$, let $E_n(I)$
be the subgroup of $E_n(R)$ generated by $E_{ij}(a)=I+ae_{ij}$ with
$a\in I$ and $1\leq i\not=j \leq n$.  Let $E_n(R,I)$ denote the normal
closure of $E_n(I)$ in $E_n(R)$.  We have two characterisation of
$E_n(R,I)$ due to Suslin-Vaserstein \cite{SV} and Stein \cite{Stein}
respectively.

\begin{proposition}\label{svs}
The kernel of the natural map $E_n(R)\ra E_n(R/I)$ is
isomorphic to $E_n(R,I)$.
\end{proposition}

\begin{proposition}\label{st}
Consider the following fiber product diagram
$$\xymatrix{ R(I)\ar [r]^{p_1} \ar [d]_{p_2} & R \ar [d]^{j_1} \\ R \ar
  [r]_{j_2} & R/I }$$ Then
$E_n(R,I)$ is kernel of
the natural surjection $E_n(p_1):E_n(R(I))\surj E_n(R)$.
\end{proposition}

Using (\ref{svs}, \ref{st},
\ref{note}) and following the proof of (\cite{AK1}, Lemma 3.3), we get
the following result. In \cite{AK1}, it is proved for Noetherian ring.

\begin{lemma}\label{ak1}
Let $A$ be a reduced ring and $P$ a projective $A$-module of rank
$r$. Assume there exist a non-zerodivisor $s\in A$ such that
(\ref{l95}) holds. Assume $E_{r+1}(B)$ acts transitively on
$\Um_{r+1}(B)$, where $B=A[X]/(X^2-s^2X)$.  Then any element of
$\Um(A\op P,s^2A)$ can be taken to $(1,0)$ by some element of $E(A\op
P)$.
\end{lemma}

The proof of the following result is same as of (\cite{AK1}, Theorem
3.4) using (\ref{ak1}, \ref{H}). 

\begin{proposition}\label{ak11}
Let $A$ be a reduced ring of dimension $d$ and $P$ a projective
$A$-module of rank $r \geq d$. Assume there exist a non-zerodivisor
$s\in A$ such that (\ref{l95}) holds. Assume $E_{r+1}(B)$ acts
transitively on $\Um_{r+1}(B)$, where $B=A[X]/(X^2-s^2X)$.  Then
$E(A\op P)$ acts transitively on $\Um(A\op P)$.
\end{proposition}

\begin{remark}\label{rl}
By (\cite{Lam}, Exercise 2.34), any reduced ring $R$ can be embedded
in a reduced non-Noetherian ring $S$ such that $S$ equals the total
quotient ring $Q(S)$ of $S$ and $R$ is a retract of $S$. In
particular, if $P$ is a non-free projective $R$-module, then $P\ot_R
S$ is a non-free projective $S$-module. Hence, if $R$ is a reduced
non-Noetherian ring and $P$ a projective $R$-module, then we can not
say that $P_s$ is free, for some non-zerodivisor $s\in R$.
\end{remark}

\begin{definition}\label{dd}
Let $R\subset S$ be rings and $P$ a projective $S$-module. We say
that $P$ satisfies property $\Omega(R)$ if for any ideal $I$ of $R$ and
$\ol P=P/IP$, there exist a non-zerodivisor $\ol t\in R/I$ such
that $\ol P_{\ol t}$ is free. The property $\Omega(R)$ avoids situation
(\ref{rl}).
\end{definition}

The following result generalises (\ref{YA}). 

\begin{theorem}\label{K3}
Let $R$ be a ring of dimension $d$ and $A$ is one of $R[Y]$ or
$R[Y,Y^{-1}]$, where $Y$ is a variable over $R$. Let $P$ be a
projective $A$-module of rank $r\geq d+1$ which satisfies property
$\Omega(R)$. Then $E(A\op P)$ acts transitively on $\Um(A\op P)$.
\end{theorem}

\begin{proof}
By (\ref{ak3.1}), we may assume $R$ is reduced.  If $d=0$, then $P$ is
free by (\ref{zero}) and we can use (\ref{YA}). Hence assume $d\geq 1$
and use induction on $d$.  Since $P$ satisfies property $\Omega(R)$,
we can find a non-zerodivisor $s\in R$ such that $P_s$ is free and
(\ref{l95}) holds.  If $R'=R[X]/(X^2-s^2X)$, then $\dim R'=d$. Write
$B=A[X]/(X^2-s^2X)$. Then $B$ is one of $R'[Y]$ or $R'[Y,Y^{- 1}]$. By
(\ref{YA}), $E_{r+1}(B)$ acts transitively on $\Um_{r+1}(B)$. Applying
(\ref{ak1}), we get every element of $\Um(A\op P,s^2A)$ can be taken
to $(1,0)$ by some element of $E(A\op P)$. Therefore it is enough to
show that every element of $\Um(A\op P)$ can be taken to an element of
$\Um(A\op P,s^2A)$ by some element of $E(A\op P)$.

Let ``bar'' denote reduction modulo $s^2A$. Then $\dim R/s^2R<d$. By
assumption, $P/s^2P$ satisfies property $\Omega(R/s^2R)$. Hence by
induction on $d$, $E(\ol A\op \ol P)$ acts transitively on $\Um(\ol
A\op \ol P)$. Using (\ref{lift}), any element of
$\Um(A\op P)$ can be taken to an element of $\Um(A\op P,s^2A)$ by
$E(A\op P)$.  This completes the proof. $\hfill \gj$
\end{proof}

\section{Some Auxiliary results}

\begin{lemma}\label{L0}
Let $R$ be a ring of dimension $d$ such that dimension of the
polynomial ring $A=R[Y_1,\ldots,Y_n]$ is $d+n$. Then every stably free
$A$-module $P$ of rank $\geq d+1$ is free.
\end{lemma}

\begin{proof}
The case $n=0$ is due to Heitmann (\ref{H}). Assume $n>0$ and use
induction on $n$. Let $S$ be the set of all monic polynomials in
$R[Y_n]$. Then $\dim R[Y_n]_S=d$ and $\dim
R[Y_n]_S[Y_1,\ldots,Y_{n-1}]=d+n-1$. Hence by induction on $n$,
$S^{-1}P$ is free. By (\ref{QS1}), $P$ is free.  $\hfill \gj$
\end{proof}

\begin{proposition}\label{L1}
Let $R$ be a ring of dimension $d$ such that dimension of the polynomial ring
$R[Y_1,\ldots,Y_n]$ is $d+n$. Let $A=R[Y_1,\ldots,Y_n,(f_1\ldots
  f_m)^{-1}]$ with $m\leq n$ and $f_i\in R[Y_i]$ a monic polynomial
for all $i$. Then every stably free $A$-module $P$ of
rank $r\geq d+1$ is free.
\end{proposition}

\begin{proof}
The case $m=0$ follows from (\ref{L0}). Assume $m>0$ and use induction
on $m$. Let $C=R[Y_1,\ldots,Y_n,(f_1\ldots f_{m-1})^{-1}]$. If
$S=1+f_mR[Y_m]$, then $\dim R[Y_m]_{f_mS}=d$ (since $\dim R[Y_m]=d+1$)
and
$S^{-1}A=R[Y_m]_{f_mS}[Y_1,\ldots,Y_{m-1},Y_{m+1},\ldots,Y_n,(f_1\ldots
  f_{m-1})^{-1}]$. By induction on $m$, $S^{-1}P$ is free. Choose
$g\in S$ such that $P_g$ is free. Patching projective modules $P$ and
$C_g^r$ over $C_{f_mg}$, we get a projective $C$-module $Q$ such that
$Q_{f_m}=P$. Since $P$ is stably free, $(Q\op C^t)_{f_m}$ is free for
some $t$. By (\ref{QS1}), $Q\op C^t$ is free, i.e. $Q$ is stably
free. By induction on $m$, $Q$ and hence $P$ is free. $\hfill \gj$
\end{proof}

It is natural to ask if all projective $A$-modules of rank $\geq d+1$ in
(\ref{L1}) are cancellative. We give a partial answer.

\begin{theorem}\label{K4}
Let $R$ be an integral domain of dimension $d$ such that $\dim
R[Y]=d+1$.  Let $A=R[Y,f^{-1}]$ with $f\in R[Y]$ and $P$ a projective
$A$-module of rank $r\geq max\{2,d+1\}$. Then $E(A\op P)$ acts
  transitively on $\Um(A\op P)$.
\end{theorem}

\begin{proof}
If $d=0$, then $P$ is free and we are done by (\ref{elem}). Assume
$d\geq 1$.  Choose $0\not=s\in R$ such that (\ref{l95}) holds. Write
$R'=R[X]/(X^2-s^2X)$ and $B=R'[Y,f^{-1}]$. Assume $E_{r+1}(B)$ acts
transitively on $\Um_{r+1}(B)$. By (\ref{ak1}), any $(a,p)\in \Um(A\op P,s^2A)$
can be taken to $(1,0)$ by some element in $E(A\op P)$. Let
``bar'' denote reduction modulo $s^2A$. Then $\dim \ol A=d$ and rank
$\ol P\geq d+1$. Applying (\ref{H}), we get $E(\ol A\op \ol P)$ acts
transitively on $\Um(\ol A\op \ol P)$. Using (\ref{lift}), every $V\in
\Um(A\op P)$ can be taken to $W\in \Um(A\op P,s^2A)$ by some element
of $E(A\op P)$. Therefore, it is enough to show that $E_{r+1}(B)$ acts
transitively on $\Um_{r+1}(B)$.

Let $v\in \Um_{r+1}(B)$. If $C=R'[Y]$, then $B=C_f$. 
Since $R'$ is an integral extension of $R$,
$\dim R[Y]=d+1=\dim R'[Y]$. Hence $\dim C_{f(1+fR'[Y])}=d$. Applying
(\ref{H}), we get $\sigma\in E_{r+1}(C_{f(1+fR'[Y])})$ such that
$\sigma(v)=(1,\ldots,0)$. We can find $g\in 1+fR'[Y]$ and $\wt \sigma
\in E_{r+1}(C_{fg})$ such that $\wt \sigma(v)=(1,0\ldots,0)$.

By (\ref{split}), $\wt \sigma$ has a splitting $\wt
\sigma=(\ga)_{f}(\gb)_g$, where $\ga\in E_{r+1}(C_g)$ and $\gb\in
E_{r+1}(C_{f})$. We have two unimodular elements $\gb(v) \in
\Um_{r+1}(C_{f})$ and $\ga^{-1}((1,0,\ldots,0))\in \Um_{r+1}(C_g)$
whose images in $C_{fg}$ are same. Hence, patching $\gb(v)$ and
$\ga^{-1}((1,0,\ldots,0))$, we get $w\in \Um_{r+1}(C)$ whose
image in $C_{f}$ is $\gb(v)$.  By Yengui (\ref{YA}), $E_{r+1}(C)$ acts
transitively on $\Um_{r+1}(C)$. Hence, we can find $\phi\in
E_{r+1}(C)$ such that $\phi(w)=(1,0,\ldots,0)$. If $\Phi_1$ is the
image of $\phi$ in $C_{f}$, then $\Phi_1(\ga(v))=(1,0,\ldots,0)$ and
$\Phi_1\ga \in E_{r+1}(B)$. This completes the proof. $\hfill \gj$
\end{proof}

\begin{remark}
$(1)$ By a result of Seidenberg (\cite{Se}, Theorem 4), if $R$ is a Pr\"ufer
domain, then $\dim R[Y_1,\ldots,Y_n]=\dim R+n$. Hence (\ref{L1}, \ref{K4}) holds
for a Pr\"ufer domain $R$ and generalizes (\ref{K3}).

$(2)$ Lequain-Simis have shown \cite{LS} that if $R$ is a Pr\"ufer
domain, then projective modules over $R[Y_1,\ldots,Y_n]$ are extended
from $R$. In particular, if $R$ is a valuation domain (local Pr\"ufer
domain), then projective $R[Y_1,\ldots,Y_n]$-modules are free. It is
natural to ask if projective modules over $R[Y_1,\ldots,Y_n,(f_1\ldots
  f_m)^{-1}]$ are free, where $R$ is a valuation domain, $m\leq n$ and
$f_i\in R[Y_i]$. If each $f_i$ is a monic polynomial, then (\ref{L1})
gives a partial answer.
\end{remark}

\begin{proposition}\label{LS1}
Let $R$ be a valuation domain of dimension $d$ and
$A=R[X,Y_1,\ldots,Y_n,f^{-1}]$ with $f\in R[X]$.  Then
every stably free $A$-module $P$ of rank $\geq d+1$ is free.
\end{proposition}

\begin{proof}
If $d=0$, then $P$ is free, by (\ref{zero}). Assume $d\geq 1$.  Let
$C=R[X,Y_1,\ldots,Y_n]$ and $S=1+fR[X]$. Since $\dim R[X]=d+1$ by
Seidenberg \cite{Se}, $ \dim R[X]_{fS}=d$ and $\dim
R[X]_{fS}[Y_1,\ldots,Y_n]=d+n$. By (\ref{L0}), $S^{-1}P$ being stably
free, is free. Choose $g\in S$ such that $P_g$ is free. Patching
projective modules $P$ and $(C_g)^r$ over $C_{fg}$, we get a
projective $C$-module $Q$ such that $P\iso Q_f$. By Lequain-Simis
\cite{LS}, every projective $C$-module is free. Therefore $Q$ and
hence $P$ is free. $\hfill \gj$
\end{proof}

\end{document}